\documentclass{amsart}

\usepackage{amsmath,amsfonts,amssymb,amsthm}
\usepackage{hyperref}
\usepackage{mathtools}
\usepackage{pdfpages}
 
\usepackage[capitalize]{cleveref}
\usepackage{stmaryrd}
\usepackage{tikz-cd}

\usepackage{thmtools}

\newtheorem{lem}{Lemma}[section]
\newtheorem{teo}[lem]{Theorem}

\newtheorem{pro}[lem]{Proposition}
\newtheorem{cor}[lem]{Corollary}
\newtheorem{claim}[lem]{Claim}

\newtheorem*{con*}{Conjecture}

\newtheorem{Conj}{Conjecture}

\newtheorem{Problem}[Conj]{Problem}

\theoremstyle{definition}

\theoremstyle{remark}
\newtheorem*{rem*}{Remark}

\newcommand{\argu}{\hbox to 7truept{\hrulefill}}

\DeclareMathOperator{\D}{\mathcal D}

\DeclareMathOperator{\tor}{tor}
\newcommand{\FL}{\mathtt{FL}}

\DeclareMathOperator{\Crit}{Crit}

\DeclareMathOperator{\rk}{rk}

\DeclareMathOperator{\Stab}{Stab}

\DeclareMathOperator{\Tor}{Tor}

\newcommand{\myeq}[1]{\ensuremath{\stackrel{\text{#1}}{=}}}
\newcommand{\myle}[1]{\ensuremath{\stackrel{\text{#1}}{\leqslant}}}

\newcommand{\myges}[1]{\ensuremath{\stackrel{\text{#1}}{>}}}

\newcommand{\myge}[1]{\ensuremath{\stackrel{\text{#1}}{\geqslant}}}

\newcommand {\brk}{\overline \rk}

\newcommand{\bpi}{\overline{\pi}}

\newcommand{\Z}{\mathbb{Z}}
\newcommand{\F}{\mathbb{F}}

\newcommand{\Q}{\mathbb{Q}}

\DeclareMathOperator{\FP}{\mathtt{FP}}

\newcounter{marcocomments}

\newcounter{andreicomments}

 \date{\today}
 
\begin{document}
\title {Free groups are $L^2$-subgroup rigid}
\author{Andrei Jaikin-Zapirain}
 \address{Departamento de Matem\'aticas, Universidad Aut\'onoma de Madrid \and  Instituto de Ciencias Matem\'aticas, CSIC-UAM-UC3M-UCM}
\email{andrei.jaikin@uam.es}
 
\begin{abstract}
In this paper, we introduce the notion of $L^2$-subgroup rigid groups and demonstrate that free groups are $L^2$-subgroup rigid. As a consequence, we establish the equivalence between compressibility, inertness, strong inertness, and $L^2$-independence for a finitely generated subgroup of a free group, confirming  a conjecture  by Dicks and Ventura as well as the one by Antolin and Jaikin-Zapirain.
\end{abstract}

\maketitle

\section{Introduction} The study of finitely generated subgroups of a free group  
 is a classical topic in combinatorial group theory. The standard approach employs Stallings' graph theory techniques introduced in \cite{St83} (see also \cite{KM02}). However, it has been realized that these techniques have limitations in addressing problems related to the intersections of subgroups, as exemplified by the case of the Hanna Neumann Conjecture. The proofs of the conjecture by Friedman\cite{Fr15}  and Mineyev\cite{Mi12}, as well as a more recent proof by Jaikin-Zapirain\cite{Ja17} use different algebraic tools. These ideas were further developed by Antolin and Jaikin-Zapirain \cite{AJ22}, where the notion of $L^2$-independence was introduced. This led  to the proof of a generalized version of the Hanna Neumann conjecture for various families of groups \cite{BK23, FM23}. In this paper, we introduce the notion of $L^2$-subgroup rigid group, providing another powerful tool for proving new results on subgroups of groups possessing this property.

Let $G$ be a group and let $\mathcal U(G)$ denote the ring of affiliated operators of $G$  (see \cref{sect:l2} for the definitions).
For a left $\mathbb{Q}[G]$-module $M$, we define
\[
\tor_{\mathbb{Q}[G]}(M) \;=\; \bigcap \{ \ker \phi \;\colon\; \phi : M \to \mathcal U(G) \ \text{is a $\mathbb{Q}[G]$-homomorphism} \}.
\]

We will show in \cref{equivdef} that, in the case where $M$ is finitely generated, $\tor_{\mathbb{Q}[G]}(M)$ can equivalently be defined as
\[
\tor_{\mathbb{Q}[G]}(M) \;=\; \bigcap \{ \ker \phi \;\colon\; \phi : M \to \ell^2(G) \ \text{is a $\mathbb{Q}[G]$-homomorphism} \}.
\]

Moreover, if $G$ is torsion-free and satisfies the strong Atiyah conjecture over $\mathbb{Q}$ (which states that the division closure $\mathcal D(G)$ of $\mathbb{Q}[G]$ in $\mathcal U(G)$ is a division ring), then
\[
\tor_{\mathbb{Q}[G]}(M) \;=\; \bigcap \{ \ker \phi \;\colon\; \phi : M \to \mathcal D(G) \ \text{is a $\mathbb{Q}[G]$-homomorphism} \}.
\]

Denote by $\iota_M$ the canonical map
\[
\iota_M : M \longrightarrow \mathcal D(G) \otimes_{\mathbb{Q}[G]} M, 
\quad m \longmapsto 1 \otimes m.
\]
Then any $\mathbb{Q}[G]$-homomorphism $\phi : M \to \mathcal D(G)$ can be canonically decomposed as 
\[
\phi = \widetilde{\phi} \circ \iota_M,
\]
where $\widetilde{\phi} : \mathcal D(G) \otimes_{\mathbb{Q}[G]} M \to \mathcal D(G)$ is a $\mathcal D(G)$-homomorphism.  
Therefore,
\[
\tor_{\mathbb{Q}[G]}(M) \;=\; \ker(\iota_M).
\]

 We denote by $I_{\Q[G]}$ the augmentation ideal of $\Q[G]$ and if $H$ is a subgroup of $G$ we denote by ${}^G I_{\Q[H]}$ 
the left ideal of $\Q[G]$ generated by $I_{\Q[H]}$.  Let $L$ be a left $\Q[G]$-submodule of $I_{\Q[G]}$ such that $L/{}^G I_{\Q[H]}=\tor_{\Q[G]}(I_{\Q[G]}/{}^GI_{\Q[H]})$. A priori, the submodule $L$  has nothing to do with the subgroup structure of $G$. 
If  there exists a subgroup $\widetilde H$  such that  $L={}^GI_{\Q[\widetilde H]}$,
then $\widetilde H$ is unique and contains $H$ (see \cref{permutationmodule}), and we call this subgroup the {\bf $L^2$-closure} of $H$ in $G$.   
We say that $G$ is {\bf $L^2$-subgroup rigid} if for all finitely generated subgroups of $G$ its $L^2$-closure can be defined.
In this paper we will show that free groups are $L^2$-subgroup rigid.
\begin{teo}\label{main}
Free groups are $L^2$-subgroup rigid.
\end{teo}
It is very plausible that our proof can be extended to show that not only free groups, but also other groups are $L^2$-subgroup rigid. We have chosen to explore this aspect in a forthcoming paper and, for simplicity, focus here only on the case of free groups. As we will see later, even in the case of free groups, we already can derive several important consequences.

A finitely generated subgroup $H$ of a free group $F$ is called {\bf compressed} if for any finitely generated subgroup $L$ of $F$ containing $H$, $\rk(H)\le \rk(L)$, and it is called {\bf inert} if for any finitely generated subgroup $L$ of $F$, $\rk(H\cap L)\le \rk(L)$. 

The concept of a subgroup being compressed is quite intuitive; it indicates that the subgroup is not contained in another subgroup of smaller rank. The concept of an inert subgroup was introduced in the work of Dicks and Ventura while they were studying equalizers in free groups \cite{DV96}.

We also say that $H$ is {\bf $L^2$-independent} in $F$  if 
 the canonical map $$\D(F)\otimes_{\Q[H]}I_{\Q[H]}\to \D(F)\otimes_{\Q[F]}I_{\Q[F]}$$ is injective. 
 It is clear that   inert subgroups are compressed and Antolin and Jaikin-Zapirain \cite{AJ22} proved that $L^2$-independent subgroups are inert. On the other hand, Dicks and Ventura \cite[Problem 1]{DV96} (see also \cite[Problem 19.102]{Kou}) asked whether  compressed subgroups are also inert and Antolin and Jaikin-Zapirain \cite[Question 2]{AJ22}  suggested that compressed  subgroups may  always be  $L^2$-independent. We confirm these two conjectures.
\begin{cor}\label{inertcompressed}
Let $F$ be a free group and $H$ a finitely generated subgroup. Then the following are equivalent:
\begin{enumerate}
\item $H$ is compressed in $F$;
\item $H$ is inert in $F$;
\item $H$ is $L^2$-independent in $F$.
\end{enumerate}
\end{cor}
In \cref{strcom} we will also show that these three properties of $H$  are also equivalent of $H$ being {\it strongly inert} in $F$ (see  \cref{l2ind} for the definition).

It is clear that the intersection of two inert subgroups is inert and by \cite[Theorem 3.9]{MVW07}, it is 	algorithmically decidable,  for a finitely generated subgroup of a free group, whether it is compressed. In \cite{Iv18} Ivanov proved   that it is algorithmically decidable whether a finitely generated subgroup of a free
group is strongly inert or not. The equivalence between inert and compressed subgroup provides us with the following immediate consequences.

\begin{cor} The following holds.
\begin{enumerate}
\item \cite[Problem 19.103]{Kou} The intersection of two compressed subgroups of a free group is again compressed.
\item   \cite[Problem 19.104]{Kou} It is algorithmically decidable, for a finitely generated subgroup of a free group, whether it is inert.
\end{enumerate}
\end{cor}

A finitely generated subgroup $H$  of a free group $F$ is called {\bf strictly compressed} in $F$ if for any subgroup $  L$ of $F$, properly containing $H$, $\rk(H)< \rk(L)$ and {\bf strictly inert}
  if for any subgroup $L$ not contained in $H$, we have
$\rk(L\cap H) < \rk(L)$. 
We say that $H$ is {\bf $L^2$-closed} in $G$ if it coincides with its own $L^2$-closure. Another consequence of \cref{main} is the following equivalence.
\begin{cor}\label{strictlycompressed}
Let $F$  a free group and $H$ a finitely generated subgroup of $F$. Then the following are equivalent.
\begin{enumerate}
\item $H$ is strictly compressed in $F$;
\item $H$ is strictly inert in $F$;
\item  $H$ is $L^2$-closed in $F$.
\end{enumerate}
\end{cor}

We also obtain  the following surprising property of the lattice of finitely generated subgroups of free groups.
\begin{cor}\label{crit}
Let $F$ be a free group and $H$ a finitely generated subgroup. Define
\begin{multline*}
\bpi(H\leq F)=\min\{\rk(L)\colon H\le L\le F\} \textrm{\ and\ }\\ \Crit(H\leq F)=\{H\leq L\leq F:  \rk(L) =\bpi(H\leq F)\}.
\end{multline*}
If  $L_1$ and $L_2$ belong to $\Crit(H\leq F)$, then $L_1\cap L_2$ and $\langle L_1,L_2\rangle$ also belong to $\Crit(H\leq F)$. \end{cor}
We will see that the maximal  subgroup in $\Crit(H\leq F)$   is the $L^2$-closure of $H$. 
  
The paper is organized as follows. In \cref{prelim}, we describe preliminary results. The main step of our proof is found in \cref{casep=2}, where we establish an analogue of \cref{main} in characteristic $p = 2$. In this section, we explain why the prime $p = 2$ is particularly suitable for our argument. The case of characteristic 0 follows as a consequence of the case $p = 2$.
In the last section, we propose several open questions for further research.
 
 \section*{Acknowledgments}This work is partially supported by grant PID2020-114032GB-I00 from the Ministry of Science and Innovation of Spain. I am very grateful to Dario Ascari, Marco Linton, Ismael Morales, Malika Roy, Pablo Sánchez, Dmitry Piontkovski, Henrique Souza, Enric Ventura, and anonymous referees for their valuable comments on earlier versions of this paper.
 \section{Preliminaries}\label{prelim}

\subsection{General results}
All of our rings $R$ will be associative and unitary. All ring homomorphisms will map $1\mapsto 1$. We will always assume that $R$ has {\bf invariant basis number} (this is, if the isomorphism of left $R$-modules $R^n\cong R^m$ implies $n = m$). The previous is ensured,
for instance, when the ring $R$ admits a ring homomorphism to a field.  We will write $\rk(M)=n$ if $M$ is a left  $R$-module isomorphic to $R^n$.

By an {\bf $R$-ring} we understand a  ring homomorphism $\varphi: R\to S$. We will often refer to $S$ as $R$-ring and omit the homomorphism $\varphi$ if $\varphi$ is clear from the context.
Two $R$-rings $\varphi_1:R\to S_1$ and $\varphi_2:R\to S_2$ are said to be {\bf isomorphic} if there exists a ring isomorphism $\alpha: S_1\to S_2$ such that $\alpha\circ \varphi_1=\varphi_2$.

If $F$ is a finitely generated free group, we denote by $\rk(F)$ the cardinality of free generators of $F$ and define $\brk(F) = \max\{\rk(F) - 1, 0\}$. In other words, $\brk(F)$ is typically $\rk(F) - 1$, except when $F$ is trivial, in which case $\brk(F) = 0$.
\subsection{Left ideals associated with subgroups}
In the introduction, we have defined the left ideal ${}^GI_{k[H]}$ of $k[G]$. 
It uniquely determines the subgroup $H$.

\begin{lem}\label{permutationmodule}
Let $H$ and $ T$ be subgroups of a group $G$ and $k$ a commutative ring.  If ${}^GI_{k[H]}\le {}^GI_{k[T]}$, then $H\le T$. In particular, if ${}^GI_{k[H]}= {}^GI_{k[T]}$, then $H=T$.
\end{lem}
\begin{proof} For any $g\in G$, we put $\overline g=gH\in G/H$. It is clear that $\Stab_G(\overline 1)=H$.
The left action of $G$ on $G/H$ induces a structure of left $k[G]$-module on $k[G/H]$. 

 Since the trivial $k[H]$-module $k$ is isomorphic to  $k[H]/I_{k[H]}$, $$k[G]/{}^GI_{k[H]}\cong k[G]\otimes _{k[H]} k  \cong k[G/H].$$ Moreover, the isomorphism is realized by the map $k[G]/{}^GI_{k[H]}\to k[G/H]$, which sends $1+{}^GI_{k[H]}$ to $\overline 1$. In particular,
$$\Stab_G(1+{}^GI_{k[H]})=\Stab_G(\overline 1)=H.$$
Thus, if ${}^GI_{k[H]}\le {}^GI_{k[T]}$, then $H=\Stab_G(1+{}^GI_{k[H]})\le \Stab_G(1+{}^GI_{k[T]})=T.$
\end{proof}
The following lemma gives an alternative description of the left $k[G]$-module ${}^GI_{k[T]}/{}^GI_{k[H]} $.
\begin{lem} \label{isomdiscr} \cite[Lemma 2.1]{Ja24}Let $H\le T$ be      subgroups of a group $G$ and $k$ a commutative ring. Then the canonical map  $$k[G]\otimes_{k[T]} (I_{k[T]}/{}^TI_{k[H]}) \to {}^GI_{k[T]}/{}^GI_{k[H]} ,$$ sending $a\otimes (b+{}^TI_{k[H]})$ to $ab+{}^GI_{k[H]}$, is an isomorphism of left $k[G]$-modules.
\end{lem}

\subsection{Algebraic subextensions}
Let $F$ be a free group and $H$ a finitely generated subgroup of $F$. For a subgroup $L$ of $F$  containing $H$, we say that $H\leq L$ is an {\bf algebraic extension} if there is no   proper free factor of $L$ containing $H$.
We put $$\mathcal A_{H\leq F}=\{L\leq F: H\leq L \textrm{\ is an algebraic extension}\}.$$ The following result was proved by  Takahasi \cite{Ta51} (see also \cite{MVW07}).
\begin{pro}\label{Afinite}
Let $F$ be a free group and $H$ a finitely generated subgroup. Then $\mathcal A_{H\leq F}$ is finite.
\end{pro}

\subsection{The $L^2$-torsion part of a left ${\Q[G]}$-module.}\label{sect:l2}
Let $G$ be a countable group and let $\ell^2(G)$ denote the Hilbert space with Hilbert basis the elements of $G$, that is, $\ell^2(G)$ consists of all square-summable formal sums
\[
\sum_{g \in G} a_g g
\]
with $a_g \in \mathbb{C}$, and inner product
\[
\left\langle \sum_{g \in G} a_g g, \sum_{g \in G} b_g g \right\rangle = \sum_{g \in G} a_g \overline{b_g}.
\]
The left and right multiplication actions of $G$ on itself extend to left and right actions of $G$ on $\ell^2(G)$. The right action of $G$ on $\ell^2(G)$ further extends to an action of elements of $\Q[G]$ on $\ell^2(G)$    as bounded linear operators on $\ell^2(G)$.

The von Neumann ring $\mathcal{N}(G)$ is the ring of bounded operators on $\ell^2(G)$ which commute with the left action of $G$. Thus, we can see $\Q[G]$ as a subalgebra of $\mathcal{N}(G)$. The ring $\mathcal{N}(G)$ satisfies the left and right Ore conditions (a result proved by S.~K.~Berberian in \cite{Be82}), and its classical ring of fractions is denoted by $\mathcal{U}(G)$. The ring $\mathcal{U}(G)$ can also be described as the ring of densely defined closed (unbounded) operators on $\ell^2(G)$ that commute with the left action of $G$ (see \cite{Lu02,Ka19}).

As mentioned in the introduction, the Strong Atiyah Conjecture predicts that if $G$ is torsion-free, then the division closure $\mathcal{D}(G)$ of $\mathbb{Q}[G]$ in $\mathcal{U}(G)$ is a division ring. 
In this paper we make use of Linnell's solution of the Strong Atiyah Conjecture for free groups~\cite{Lin93}. We also note that it was solved for all locally indicable groups in \cite {JL19}.

For a left $\mathbb{Q}[G]$-module $M$, we define
\[
\tor_{\mathbb{Q}[G]}(M) \;=\; \bigcap \{ \ker \phi \;\colon\; \phi : M \to \mathcal U(G) \ \text{is a $\mathbb{Q}[G]$-homomorphism} \}.
\]
 
\begin{lem}\label{equivdef}
Let $G$ be a group and $M$ a    finitely generated left $\Q[G]$-module. 

\begin{enumerate}
    \item[(1)] If $M$   is finitely generated, then
    \[
\tor_{\mathbb{Q}[G]}(M) \;=\; \bigcap \{ \ker \phi \;\colon\; \phi : M \to \ell^2(G) \ \text{is a $\mathbb{Q}[G]$-homomorphism} \}.
\]
\item [(2)] if $G$ is torsion-free and satisfies the strong Atiyah conjecture over $\mathbb{Q}$, then
\[
\tor_{\mathbb{Q}[G]}(M) \;=\; \bigcap \{ \ker \phi \;\colon\; \phi : M \to \mathcal D(G) \ \text{is a $\mathbb{Q}[G]$-homomorphism} \}.
\]
\end{enumerate}
\end{lem}
\begin{proof}
(1)  Since $\mathcal{U}(G)$ is the classical Ore localization of $\mathcal{N}(G)$, for any finite set of elements $m_1, \ldots, m_k \in \mathcal{U}(G)$ there exists an invertible element $u \in \mathcal{N}(G)$ such that 
\[
m_1 \cdot u, \ldots, m_k \cdot u \in \mathcal{N}(G).
\]
Since $u$ is invertible, the left $\mathbb{Q}[G]$-submodules of $\mathcal{U}(G)$ generated by $\{m_1, \ldots, m_k\}$ and by $\{m_1 \cdot u, \ldots, m_k \cdot u\}$ are isomorphic. Thus, any finitely generated submodule of $\mathcal{U}(G)$ is also a submodule of $\mathcal{N}(G)$.

Observe that the map
\[
\mathcal{N}(G) \longrightarrow \ell^2(G), \quad \phi \mapsto \phi(1),
\]
is an embedding (and moreover a $\mathbb{Q}[G]$-homomorphism). On the other hand, for any $u \in \ell^2(G)$, we can define a densely defined closed (unbounded) operator on $\ell^2(G)$, $\phi_u$ that sends $g$ to $g u$. This provides an embedding of $\ell^2(G)$ into $\mathcal{U}(G)$. 

We conclude that the local structures of the left $\mathbb{Q}[G]$-modules $\ell^2(G)$ and $\mathcal{U}(G)$ coincide, which implies the claim since $M$ is finitely generated.

 (2) Since $\mathcal{D}(G)$ is a $\mathbb{Q}[G]$-submodule of $\mathcal{U}(G)$, it is clear that
\[
\tor_{\mathbb{Q}[G]}(M) \;\subseteq\; \bigcap \left\{ \ker \phi \;\colon\; \phi : M \to \mathcal{D}(G) \ \text{is a $\mathbb{Q}[G]$-homomorphism} \right\}.
\]
To prove the converse inclusion, consider $m \in M$ and a $\mathbb{Q}[G]$-homomorphism 
\[
\phi : M \longrightarrow \mathcal{U}(G)
\]
such that $\phi(m) \neq 0$. Since $\mathcal{D}(G)$ is a division ring and $\mathcal U(G)$ is a $\D(G)$-space, there exists a $\mathcal{D}(G)$-homomorphism 
\[
\psi : \mathcal{U}(G) \longrightarrow \mathcal{D}(G)
\]
with $\psi(\phi(m)) \neq 0$. Therefore,
\[
m \notin \bigcap \left\{ \ker \phi \;\colon\; \phi : M \to \mathcal{D}(G) \ \text{is a $\mathbb{Q}[G]$-homomorphism} \right\}.\qedhere
\]
 \end{proof}

\subsection{Universal division ring of fractions of group rings}\label{sect:universal}
Let $k$ be a commutative domain and  $G$ a group. A {\bf division ${k[G]}$-ring of fractions} is an embedding $k[G]\hookrightarrow \D$ of ${k[G]}$  into a division ring $\D$ such that the elements of ${k[G]}$ generate $\D$ as a division ring. If $G$ satisfies the Strong Atiyah Conjecture, then the embedding 
\[
\mathbb{Q}[G] \hookrightarrow \mathcal{D}(G)
\]
is an example of a division $\mathbb{Q}[G]$-ring of fractions.

In this paper we will work with a division ${k[G]}$-ring of fractions $\mathcal{D}$ 
satisfying the following additional properties:

\begin{enumerate}
\item[(P1)] For every division ${k[G]}$-ring $\mathcal E$, and for  any  left finitely presented ${k[G]}$-module $M$,
$\dim_{\mathcal D}  \mathcal D\otimes _{k[G]} M\le \dim_{\mathcal E} \mathcal E \otimes _{k[G]} M.$
In this case we say that the embedding ${k[G]}\hookrightarrow \mathcal{D}$ 
is the {\bf universal division ${k[G]}$-ring of fractions}.  
By definition, it is uniquely determined up to ${k[G]}$-isomorphism 
(see~\cite{Ma80}); to emphasize this, we denote $\mathcal{D}$ by $\mathcal{D}_{k[G]}$.  
If $Q$ is the field of fractions of $k$, then $\mathcal{D}_{k[G]}$ is automatically universal for $Q[G]$; 
in other words, 
\(
\mathcal{D}_{Q[G]} \cong \mathcal{D}_{k[G]}
\)
as $Q[G]$-rings. We define
$$\beta_i^{k[G]}(M)=\dim_{\mathcal D_{k[G]}} \Tor_i(\mathcal D_{k[G]}, M).$$

\item[(P2)] $\D_{Q[G]}$ is of weak dimension 1 as a right ${Q[G]}$-module. This means  that for any  left ${Q[G]}$-module $M$,
$\beta_i^{Q[G]}(M)=0$ for $i>1$.
\item[(P3)] For any left ${Q[G]}$-submodule $M$ of $\D_{Q[G]}$,  $\beta_1^{Q[G]}(M)=0$.
\item[(P4)]  The division ${k[G]}$-ring $\D_{{k[G]}}$ is {\bf Lewin}, i.e.
for any subgroup $H$ of $G$, if we denote by $\D_H$  the division closure of $k[H]$ in $\D_{k[G]}$, then  the map $$\D_H\otimes_{k[H]} k[G]\to \D_{k[G]},\ d\otimes   a\mapsto da \ (d\in \D_H, a\in k[H]),$$ is injective.
\end{enumerate}
If $k$ is a field or $\Z$ and $F$ is a free group,  then the group ring $k[F]$ has an embedding in a division ring  satisfying  the properties (P1)-(P3), because it is a Sylvester domain (see \cite{DS78}). Moreover, it also satisfies (P4) (see, for example, \cite[Subsection 2.3]{Ja21}). 

We will need the following consequence of Property (P1) for $\Z[F]$. Let  $M$ be  a left  finitely presented   $\Z[F]$-module. Since the embedding $\Z[F]\hookrightarrow \D_{\Z[F]}$ is universal and $\D_{\Z[F]}=\D_{\Q[F]}$,
$$\dim_{\D_{\Z[F]}}\D_{\Z[F]}\otimes_{\Z[F]} M\le 
\dim_{\D_{\F_p[F]}}\D_{\F_p[F]}\otimes_{\Z[F]} M$$
and  we conclude that 
\begin{multline}\label{Qp}
\beta_0^{\Q[F]}(\Q\otimes_{\Z} M)=\dim_{\D_{\Q[F]}}\D_{\Q[F]}\otimes_{\Q[F]} (\Q\otimes_{\Z} M)=\\ \dim_{\D_{\Z[F]}}\D_{\Z[F]}\otimes_{\Z[F]} M\le 
\dim_{\D_{\F_p[F]}}\D_{\F_p[F]}\otimes_{\Z[F]} M=\\ \dim_{\D_{\F_p[F]}}\D_{\F_p[F]}\otimes_{\F_p[F]} (\F_p\otimes_{\Z} M)= \beta_0^{\F_p[F]}(\F_p\otimes_{\Z} M).
\end{multline}
Given an extension of fields $K_1\le K_2$, we also have that the division closure of $K_1[F]$ in $\D_{K_2[F]}$ is isomorphic (as a $K_1[F]$-ring) to  $\D_{K_1[F]}$. 
Thus, we obtain the following immediate consequence.
\begin{pro}\label{extfields} Let  $K_1\le K_2$ be an extension of fields and $M$ a
left  $K_1[F]$-module.  Then 
for any $i\geqslant  0$,
$
\beta_i^{K_2[F]}(K_2\otimes_{K_1} M)=\beta_i^{K_1[F]}(M).
$
\end{pro}

  Let $K$ be a field. The property (P4) has its origin in a work of Hughes  \cite{Hu70}. The result of Hughes implies that   if $G$ is locally indicable,  then there exists at most one division $K[G]$-ring of fractions satisfying (P4). In the case of free  group $F$ this has two consequences. First,  
  for any subgroup $H$ of $F$, $\D_H$ is isomorphic to $\D_{K[H]}$ as a $K[H]$-ring  and second that 
  the division $\Q[F]$-ring $\D(F)$ mentioned in the introduction is isomorphic to $\D_{\Q[F]}$ as a $\Q[F]$-ring. Thus, for any left $\Q[F]$-module $M$ we can define its {\bf $i$th $L^2$-Betti numbers} in a pure algebraic way as
$\beta_i^{\Q[F]}(M)$. By analogy, if $M$ is a left $\F_p[F]$-module $\beta_i^{\F_p[F]}(M)$ are called
the {\bf $i$th mod-$p$ $L^2$-Betti numbers} of $M$.

\subsection{The Euler characteristic}\label{sectFL}

A $R$-module $M$ is of type $\FL$  (resp. $\FP$) if it has a finite resolution consisting of finitely generated free (resp. projective) $R$-modules.

For a $R$-module $M$ of type $\FL$ that has the following  resolution 
\begin{equation*}
     \begin{tikzcd}
         0  \ar[r]&  R^{n_k}  \ar[r] &  \cdots \ar[r] &  R^{n_0} \ar[r] & M  \ar[r] & 0,
\end{tikzcd}
\end{equation*} 
 we define the Euler characteristic of $M$ as  $$\chi^R(M)=\sum_{i=0}^k (-1)^i n_i.$$ A standard aplication of Schanuel’s lemma \cite[Chapter VIII, Lemma 4.4]{Bro82} implies, that this definition does not depend on the free resolution, and, noreover, if $0\to M_1\to M_2\to M_3\to 0$ is an exact sequence of   $FL$ $R$-modules, then 

\begin{equation}
    \label{FL}
    \chi^R(M_2)=\chi^R(M_1)+\chi^R(M_3).
\end{equation}
If $R\to \D$ is a division $R$-ring, the $\D$-Betti numbers of a left $R$-module $M$ are defined as 
\[b^{R,\D}_i(M)=\dim_{\D}  \Tor_i^{R}(\D,M).\]
 If $M$ is of type $\FL$, then using  the long exact sequence of $\Tor$-functors one obtains that the Euler characteristic of $M$ can be calculated as
\begin{equation}
\label{calculatingEuler}
\chi^{R}(M)=\sum_{i} (-1)^i b^{R,\D}_i(M).
\end{equation}

\begin{pro}\label{FP} \cite[Proposition 1.4 and Proposition 4.1b]{Bie81}
 Let $0\to M_1\to M_2\to M_3\to 0$ be an exact sequence of  $R$-modules. If two of the modules $M_1$, $M_2$ or $M_3$ are of type $\FP$, then  the third one is also of type $\FP$. 
\end{pro}
The same statement can also be obtained for $\FL$-modules using \cite[Theorem~3.9.1]{Bou80}. However, we will not use this fact, since all modules of type $\FP$ considered in this paper are also of type $\FL$.

In the case of rings $k[G]$ satisfying the properties (P1)-(P3) we have the following result.

\begin{pro}\label{positiveEuler} 
Assume $R=k[G]$ be a group ring having an embedding in a division ring satisfying the properties (P1)-(P3) and let $M$ be a left $\FL$ $R$-module. Then
$$\chi^R(M)=\beta_0^R(M)-\beta_1^R(M).$$
In particular, for any non-trivial   left $\FL$ $R$-submodule $M$ of $\D_R$,  $\chi^R(M)>0$.
\end{pro}
\begin{proof}From \cref{calculatingEuler} we obtain that
\[
\chi^R(M) = \sum_{i=0}^k (-1)^i \, \beta_i^R(M).
\]
It then follows from Property (P2) that
\[
\chi^R(M) = \beta_0^R(M)-\beta_1^R(M),
\]
and from Property (P3) that if  $M$ is a left   $R$-submodule  of $\D_R$, then 
\[
\chi^R(M) = \beta_0^R(M).
\]
If  $M$ is a non-trivial submodule of $\mathcal{D}_R$, we have $\beta_0^R(M)=\dim_{\D_R} \D_R\otimes _{R} M \geq 1$.
\end{proof}

\subsection{Modules over group algebra over a free group}
 
\begin{pro}\label{FLZF}   Let $K$ be a field and $F$ a free group.   Then the ring $K[F]$ has global dimension~$1$, and   every finitely presented $K[F]$-module  is of type $\FL$.

 Moreover, if $M$ is   a finitely generated $K[F]$-module and $\beta_1^{K[F]}(M)<\infty$, then there exists an exact sequence
$$0\to {K[F]}^r\to {K[F]}^d\to M\to 0.$$
\end{pro}
\begin{proof}
The first part follows from \cite[Corollary~5]{Co64}, where it is shown that every submodule of a free $K[F]$-module is free. 

For the second part of the proposition, we have an exact sequence 
\[
0 \longrightarrow L \longrightarrow K[F]^d \longrightarrow M \longrightarrow 0,
\]
where $L$ is a free $K[F]$-module. The condition $\beta_1^{K[F]}(M)<\infty$ implies that $\beta_0^{K[F]}(L)<\infty$, and so, $L$ is finitely generated.

\end{proof}
As a corollary of the previous proposition, we obtain the following result.            
\begin{lem}\label{0tor1} Let $K$ be a field, $F$ a free group, $M_1\le M_2$ two left $K[F]$-modules and $N$ a right  $K[F]$-module. Then  the natural map 
$$\Tor_1^{K[F]}(N,M_1)\to \Tor_1^{K[F]}(N,M_2)$$ is injective.
\end{lem}
\begin{proof}
Using the long exact sequence for $\Tor$ functors, we obtain that the kernel of this map is $\Tor_2^{K[F]}(N, M_2 / M_1)$, which is zero since $K[F]$ has global dimension 1.
\end{proof}

\subsection{$\D_{K[F]}$-independent modules}\label{l2ind}
Let $F$ be a free group and $K$ a field.
By analogy with $L^2$-independence, we also say that a subgroup $H$ is {\bf $\D_{K[F]}$-independent} in  $F$ if the canonical map $$\D_{K[F]}\otimes_{K[H]}I_{K[H]}\to \D_{K[F]}\otimes_{K[F]}I_{K[F]}$$ is injective. This is equivalent to the condition $\beta_1^{K[F]}(I_{K[F]}/{}^FI_{K[H]})=0$. Observe that, since $\mathcal D(F)$ is isomorphic as a $\Q[F]$-ring  to $\D_{\Q[F]}$,
$\D_{\Q[F]}$-independence is the same as $L^2$-independence defined before.

If $H$ is a finitely generated subgroup of a free group $F$, we say that $H$ is {\bf strongly inert} in $F$ if for any finitely generated subgroup $U$ of $F$ we have that
   $$\sum_{x\in U \backslash  F/H} \brk (U\cap xHx^{-1})\le   \brk(U) .$$
This notion has been itroduced by Ivanov \cite{Iv18}. It is clear that a strongly inert subgroup of  a free group is also inert.
 The proof of the following proposition is a small modification of the proof of   \cite[Proposition 5.2]{AJ22}. 
 \begin{pro}
Let $H$ be a finitely generated subgroup of a finitely generated free group $F$. Assume that   $H$ is { $\D_{K[F]}$-independent} in $F$. Then $H$ is strongly  inert in $F$.\end{pro}
\begin{proof}
Let $U$ be a finitely generated subgroup of $F$.  Since $H$ is { $\D_{K[F]}$-independent} in $F$ 
$$
\Tor_1^{K[F]}( \D_{K[F]}, I_{K[F]}/{}^FI_{K[H]})=0.
$$
By Property (P4), the right  $K[F]$-module $   \D_{K[U]}  \otimes_{K[U]}K[F]$ naturally embeds into $\D_{K[F]}$. Hence, by \cref{0tor1}, we also have 
$$\Tor_1^{K[F]}(\D_{K[U]} \otimes_{K[U]} K[F]  , I_{K[F]}/{}^F I_{K[H]}  )=0.$$
Thus, by Shapiro's Lemma 
\begin{multline}\label{zerotor1}
\Tor_1^{K[U]}(     \D_{K[U]},  I_{K[F]}/{}^FI_{K[H]}  )=\\ \Tor_1^{K[F]}(  \D_{K[U]} \otimes_{K[U]} K[F], I_{K[F]}/{}^FI_{K[H]}   )
=0.
\end{multline}
As we have shown in the proof of \cref{permutationmodule}, the left $K[F]$-modules  $K[F]/{}^FI_{K[H]}$ and $K[F/H]$ are isomorphic. 
Thus, we have the following isomorphism of left $K[U]$-modules.
$$ K[F]/{}^FI_{K[H]}\cong K[F/H]\cong  \bigoplus_{x\in U \backslash  F/H} K[U/(U\cap xHx^{-1})].$$ Since, $\brk(U\cap xHx^{-1})=\dim_{\D_{K[U]}}
 \Tor_1^{K[U]}( \D_{K[U]}, K[U/(U\cap xHx^{-1})])$, we obtain that
 $$\sum_{x\in U \backslash  F/H} \brk (U\cap xHx^{-1})=\dim_{\D_{K[U]}}\Tor_1^{K[U]}( \D_{K[U]}, K[F]/{}^FI_{K[H]}  ).$$
 From the exact sequence of the left $K[U]$-modules
 $$0\to  I_{K[F]}/{}^FI_{K[H]}  \to K[F]/{}^FI_{K[H]}\to K\to 0,$$ we obtain that
 \begin{multline*}\dim_{\D_{K[U]}}\Tor_1^{K[U]}( \D_{K[U]}, K[F]/{}^FI_{K[H]}  )\le \\
 \dim_{\D_{K[U]}}\Tor_1^{K[U]}( \D_{K[U]},  I_{K[F]}/{}^FI_{K[H]}  )+ \dim_{\D_{K[U]}}\Tor_1^{K[U]}( \D_{K[U]},  K  )
 \myeq{\cref{zerotor1}}\\ \dim_{\D_{K[U]}}\Tor_1^{K[U]}( \D_{K[U]},  K  )= \brk(U). 
 \end{multline*}
  This gives the desired inequality:
  $$\sum_{x\in U \backslash  F/H} \brk (U\cap xHx^{-1})\le \brk(U).  \qedhere$$.
\end{proof}

\section{Proof of the results}

In this section we assume that $K$ is a field  and $F$ is a free group.

\subsection{The $\D_{K[F]}$-torsion part of a left ${K[F]}$-module.}
Let $k$ be either a field $K$ or $k = \mathbb{Z}$, and let $M$ be a left $k[F]$-module. Denote by 
\(
\tor_{k[F]}(M)
\)
the kernel of the canonical map 
\[
M \longrightarrow \mathcal{D}_{k[F]} \otimes_{k[F]} M.
\]

 The next result shows that $\tor_{K[F]}(M)$ is of type $\FL$  if $M$ is of type $\FL$.

\begin{pro}\label{torfp}
Let   $M$ be  a left finitely presented $K[F]$-module. Then  $\tor_{K[F]}(M)$ is of type $\FL$.  \end{pro}
\begin{proof}
By Property (P3), $\beta^{K[F]}_1(M/\tor_{K[F]}(M))=0$. Hence, by \cref{FLZF}, $M$ and $M/\tor_{K[F]}(M)$ are of type $\FL$. Hence, by \cref{FP}, $\tor_{K[F]}(M)$ is of type $\FP$ and so it is of type $\FL$, by \cref{FLZF}.
\end{proof}

The following result gives a characterization of $\tor_{K[F]}(M)$.
\begin{pro}\label{proptor}
Let   $M$ be a left finitely presented   ${K[F]}$-module. Then we have that $$\beta_0^{{K[F]}}(\tor_{{K[F]}}(M))=0.$$
 Moreover, $\tor_{{K[F]}}(M)$ is the maximal submodule of $M$ with this property and we also have that $$ \beta_1^{{K[F]}}(\tor_{{K[F]}}(M))=\beta_1^{{K[F]}}(M),$$ and so, $ \chi^{{K[F]}}(\tor_{{K[F]}}(M))=-\beta_1^{{K[F]}}(M)$. 
\end{pro}
\begin{proof}
Let $M_1 = \operatorname{tor}_{K[F]}(M)$. By \cref{torfp}, it is a left $K[F]$-module of type $\mathsf{FL}$. 
Observe that 
\[
\beta_0^{K[F]}(M) = \beta_0^{K[F]}(M/M_1),
\]
and by property  (P3), we also have $\beta_1^{K[F]}(M/M_1) = 0$. 
The short exact sequence
\[
0 \longrightarrow M_1 \longrightarrow M \longrightarrow M/M_1 \longrightarrow 0
\]
induces the long exact sequence in $\Tor$:
\begin{multline*}
0 \rightarrow \Tor_{1}^{\D_{K[F]}}(\D_{K[F]}, M_1) \rightarrow 
\Tor_{1}^{\D_{K[F]}}(\D_{K[F]}, M) \rightarrow 
\Tor_{1}^{\D_{K[F]}}(\D_{K[F]}, M/M_1) \rightarrow \\
\Tor_{0}^{\D_{K[F]}}(\D_{K[F]}, M_1) \rightarrow
\Tor_{0}^{\D_{K[F]}}(\D_{K[F]}, M) \rightarrow
\Tor_{0}^{\D_{K[F]}}(\D_{K[F]}, M/M_1) \rightarrow 0.
\end{multline*}
Hence, we obtain
\[
\beta_0^{K[F]}(M_1) = 0 \quad \text{and} \quad 
\beta_1^{K[F]}(M_1) = \beta_1^{K[F]}(M).
\]
Finally, it is clear that if $\beta_0^{K[F]}(N) = 0$, then $N \leq M_1$.

\end{proof}

 \subsection{$\D_{K[F]}$-torsion-free modules}
 We say that $M$ is {\bf $\D_{K[F]}$-torsion-free} if $$\tor_{K[F]}(M)=0.$$ 
 The following result is similar to \cite[Lemma 4.7]{Ja24}.
 \begin{pro}
 \label{extabs}  Let $K$ be a field,  $H$   a subgroup of a free group $F$ and
 $M$  a $\D_{K[H]}$-torsion-free left $K[H]$-module. Then $K[F]\otimes_{K[H]} M$ is $\D_{K[F]}$-torsion-free.
 \end{pro}
\begin{proof}   Let $\D_H$ be the division closure of $K[H]$ in $\D_{K[F]}$. As we mentioned in \cref{sect:universal}, $\D_H$ and $\D_{K[H]}$ are isomorphic as  $K[H]$-rings.
Therefore,  the map $ M\to \D_H\otimes_{K[H]}M$ is injective. Then, since $K[F]$ is a free right $K[H]$-module, the map
$$K[F]\otimes_{K[H]}M\xrightarrow {\alpha}K[F]\otimes_{K[H]}( \D_H\otimes_{K[H]} M)$$ is also injective. 

Consider  the canonical isomorphism between tensor products 
$$K[F]\otimes_{K[H]}( \D_H\otimes_{K[H]} M)\xrightarrow {\beta} (K[F]\otimes_{K[H]} \D_H)\otimes_{K[H]} M.$$ 
By Property (P4),
 the canonical map    $$K[F]\otimes_{K[H]} \D_H\to \D_{K[F]}$$ is injective. Moreover, since $\D_H$ is a division ring, the image of $K[F]\otimes_{K[H]} \D_H$  is a direct    summand of $\D_{K[F]}$ as a right $\D_H$-submodule (and so, it is also a direct  summand as a right $K[H]$-submodule). Thus,  the following canonical map  of $K[F]$-modules
 $$ (K[F]\otimes_{K[H]} \D_H)\otimes_{K[H]}M \xrightarrow {\gamma}  \D_{K[F]}\otimes_{K[H]} M$$ 
 is injective. 
Hence $\varphi=\gamma\circ \beta\circ \alpha$ is an injective homomorphism of left $K[F]$-modules. Since a submodule of a $\D_{K[F]}$-torsion-free module is also  $\D_{K[F]}$-torsion-free,
 $K[F]\otimes_{K[H]} M$ is $\D_{K[F]}$-torsion-free.
\end{proof}

 Observe that if $K_1$ is a subfield of a field $K_2$, then the division closure of $K_1[F]$ in $\D_{K_2[F]}$ is isomorphic to $\D_{K_1[F]}$ as a $K_1[F]$-ring. This implies the following consequence.
\begin{lem} \label{extcuerpostf} 
Let $K_1$ be  a subfield of a field $K_2$ and let $F$ be a free group. If a $K_1[F]$-module $M$ is $\D_{K_1[F]}$-torsion-free, then the $K_2[F]$-module $K_2\otimes_{K_1} M$ is $\D_{K_2[F]}$-torsion-free.
\end{lem}

\subsection{$\D_{\F_2[F]}$-subgroup rigidity of free groups} \label{casep=2}
We say that $F$ is $\D_{K[F]}$-{\bf subgroup rigid} if for any finitely generated subgroup $H$ of $F$, there exists a subgroup $\widetilde H$ of $F$ containing $H$ such that 
$$\tor_{K[F]}(I_{K[F]}/{}^FI_{K[H]})={}^FI_{K[\widetilde H]}/{}^FI_{K[H]}.$$ In this section we prove the following analog of \cref{main}.
\begin{teo}\label{mainF2} Let $F$ be a free group. Then $F$ is
 $\D_{\F_2[F]}$-subgroup rigid.
\end{teo}
First we will prove  an auxiliary result.
It describes the $K[F]$-submodules of $I_{K[F]}$ of codimension 1 in the case $K=\F_2$. This is where we strongly use the condition $K=\F_2$.
\begin {lem}\label{maxsubm} Let $F$ be a finitely generated free group.
Let $N$ be a left $\F_2[F]$-submodule of $I_{\F_2[F]}$ such that $I_{\F_2[F]}/N\cong \F_2$ is the trivial $\F_2[F]$-module of dimension 1 over $\F_2$. Then there exists a subgroup $U$ of $F$ of index 2 such that $N={}^FI_{\F_2[U]}$.
\end{lem}
\begin{proof}Consider the set $\mathcal S$ of left $\F_2[F]$-submodules $N$ of $I_{\F_2[F]}$ such that $I_{\F_2[F]}/N\cong \F_2$. Then $\mathcal S$ consists of the kernels of non-trivial $\F_2[F]$-homomorphisms from $I_{\F_2[F]}$ to $\F_2$. Since, $\rk(I_{\F_2[F]})=\rk(F)$, $\mathcal S$ has $2^{\rk(F)}-1$ submodules. If $U$ is a subgroup of  $F$ of index 2, $\dim_{\F_2}\F_2[F]/{}^FI_{\F_2[U]}=2$, and so, $\dim_{\F_2}I_{\F_2[F]}/{}^FI_{\F_2[U]}=1$. Thus, ${}^FI_{\F_2[U]}\in S$. Since, the number of subgroups of  $F$ of index 2 is $2^{\rk(F)}-1$, we obtain the lemma.
\end{proof}

\begin{proof}[Proof of \cref{mainF2}] We start the proof with the following useful claim.
\begin{claim} \label{indtor} 
Let $H \leq U \leq L \leq F$.  
Then the $\mathbb{F}_2[L]$-module 
\[
\tor_{\mathbb{F}_2[L]}\!\left({}^L I_{\mathbb{F}_2[U]} \,\big/\, {}^L I_{\mathbb{F}_2[H]}\right)
\]
is equal to the $\mathbb{F}_2[L]$-submodule of 
$
{}^L I_{\mathbb{F}_2[U]} \,\big/\, {}^L I_{\mathbb{F}_2[H]}
$
generated by 
\[
\tor_{\mathbb{F}_2[U]}\!\left(I_{\mathbb{F}_2[U]} \,\big/\, {}^U I_{\mathbb{F}_2[H]}\right).
\]
\end{claim}

\begin{proof} Recall  that $\D_U$ denotes the division closure of $\F_2[U]$ in $\D_{\F_2[F]}$. Since $\D_U$ is isomorphic to $\D_{\F_2[U]}$ as a $\F_2[U]$-ring, we have that 
$$\tor_{\F_2[U]} ( I_{\F_2[U]}/{}^UI_{\F_2[H]})\le \tor_{\F_2[L]}({}^LI_{\F_2[U]}/{}^LI_{\F_2[H]}).$$ 
On the other hand, let $N_0= I_{\F_2[U]}/{}^UI_{\F_2[H]}$,  $N={}^LI_{\F_2[U]}/{}^LI_{\F_2[H]}$ and $N_1$ be  the $\F_2[L]$-submodule of $N$ generated by 
$\tor_{\F_2[U]}(N_0)$.
By \cref{isomdiscr}, we know that $N\cong \F_2[L]\otimes_{\F_2[U]} N_0$. Hence $N/N_1$ is isomorphic to    $\F_2[L]\otimes_{\F_2[U]} (N_0/\tor_{\F_2[U]}(N_0))$. By  \cref{extabs}, $N/N_1$  is $\D_{\F_2[U]}$-torsion-free. Hence 
$\tor_{\F_2[L]}(N)\le N_1.$ 
\renewcommand\qedsymbol{$\diamond$}\end{proof}

Let   $H$ be a finitely generated subgroup of $F$. Define
\begin{align*}\mathcal S=  \{H\le L\le F: & \textrm{ there exists no subgroup $\widetilde H\le L$ containing $H$ such that}  \\ & \tor_{\F_2[L]}(I_{\F_2[L]}/{}^LI_{\F_2[H]}) 
\textrm{ is equal 
to\ } 
{}^L I_{\F_2[\widetilde H]}/{}^LI_{\F_2[H]} \}.\end{align*}
\begin{claim}\label{nonex} Let $H\le L\le F$ and let $U$ be a free factor  of $L$  containing $H$. If $L\in \mathcal S$, then $U\in \mathcal S$.  
In particular,  if $\mathcal S\ne \emptyset$, then  $\mathcal S\cap  \mathcal A_{H\leq L}\neq \emptyset$.  
\end{claim}
\begin{proof}  \renewcommand\qedsymbol{$\diamond$}   Let $U\leq L$ such that $L=U*U_1$. Then 
  $$I_{\F_2[L]}/{}^LI_{\F_2[H]}=\left ({}^{L}I_{\F_2[U]}/{}^LI_{\F_2[H]}\right ) \oplus {}^LI_{\F_2[U_1]}.$$
  Hence $\tor_{\F_2[L]}(I_{\F_2[L]}/{}^LI_{\F_2[H]})=\tor_{\F_2[F]}({}^{L}I_{\F_2[U]}/{}^LI_{\F_2[H]})$, and so by \cref{indtor}, $U\in \mathcal S$.
\renewcommand\qedsymbol{$\diamond$}\end{proof}
Assume that  $F\in \mathcal S$. We want to obtain a contradiction.
 By \cref{nonex}, $ \mathcal S\cap  \mathcal A_{H\leq F} \neq \emptyset$.  By \cref{Afinite} there exists  a minimal subgroup $L$ in $\mathcal S\cap  \mathcal A_{H\leq F}$. In particular $L$ is finitely generated.

Put $M=I_{\F_2[L]}/{}^LI_{\F_2[H]}$. Since $L\in \mathcal S$, 
 $\tor_{\F_2[L]}(M)\ne \{0\}$.
  \begin{claim} \label{subgroupindex2}
There exists a  subgroup $U$ of  $L$ of  index 2 such that $$\tor_{\F_2[L]}(M)\leq {}^LI_{\F_2[U]}/{}^LI_{\F_2[H]}.$$ 
\end{claim}
\begin{proof} Since $L\in \mathcal S$, 
 $\tor_{\F_2[L]}(M)\ne M$. Therefore, $$\beta_0^{\F_2[L]}(M/\tor_{\F_2[L]}(M))>0.$$ 
Therefore, by Property (P1), $\F_2\otimes_{\F_2[L]}(M/\tor_{\F_2[L]}(M))$ is not trivial. Hence there exists  a left $\F_2[L]$-submodule $N$ of $M$ such that 
 $\tor_{\F_2[L]}(M)\leq N  <M$ and $M/N \cong \F_2$. By \cref{maxsubm}, $N ={}^LI_{\F_2[U]}/{}^LI_{\F_2[H]}$ for some subgroup $U$ of index 2 in $L$. 
\renewcommand\qedsymbol{$\diamond$}\end{proof}
Take $U$ from the previous claim. From \cref{proptor} we obtain that
\begin{equation}
\label{eqtor}
\tor_{\F_2[L]}(M)=\tor_{\F_2[L]} \left ({}^LI_{\F_2[U]}\big /{}^LI_{\F_2[H]}\right ).\end{equation}

By our choice of $L$ and \cref{nonex}, there exists a subgroup $\widetilde H$ of $U$ containing $H$, such that 
\begin{equation}\label{limittor}
\tor_{\F_2[U]}(I_{\F_2[U]}/{}^UI_{\F_2[H]})={}^UI_{\F_2[\widetilde H]}/{}^UI_{\F_2[H]}.
\end{equation}
Hence,
$$\tor_{\F_2[L]}(M)\myeq{\cref{eqtor}}\tor_{\F_2[L]}({}^LI_{\F_2[U]}/{}^LI_{\F_2[H]})\myeq{\cref{indtor} and \cref{limittor}}{}^LI_{\F_2[\widetilde H]}/{}^LI_{\F_2[H]}.$$
This is a contradiction that finishes the proof of the theorem.
\end{proof}

Let $K$ be  a field and $H$ a subgroup of a free group $F$. Assume that there exists a subgroup $\widetilde H$ of $F$ containing $H$  such that 
$$\tor_{K[F]}(I_{K[F]}/{}^FI_{K[H]})={}^FI_{K[\widetilde H]}/{}^FI_{K[H]}.$$
By \cref{permutationmodule}, $\widetilde H$ is unique and 
 we say that $\widetilde H$ is  the {\bf $\D_{K[F]}$-closure} of $H$. By \cref{equivdef}, the notion of $L^2$-closure introduced in the introduction and the notion of $\mathcal{D}_{\mathbb{Q}[F]}$-closure coincide.

 \cref{mainF2} implies that for every finitely generated subgroup $H$ of $G$ the $\D_{\F_2[F]}$-closure of $H$ can be defined. 
In fact, we can also conclude that the $\D_{\F_2[F]}$-closure of an arbitrary subgroup $H$ of $F$ can be defined. Indeed, let $H = \bigcup H_i$ be expressed as a union  of finitely generated subgroups. Then the $\D_{\F_2[F]}$-closure of $H$ coincides with the union of the $\D_{\F_2[F]}$-closures of the subgroups $H_i$.

By \cref{extcuerpostf}, if $K$ is a field of characteristic 2, then 
$\D_{K[F]}$-closure of $H$ coincides with its $\D_{\F_2[F]}$-closure.

\subsection{Proof of \cref{main}}
Let $F$ be a free group and $H$ a finitely generated subgroup of $F$. Let $\widetilde H$ be the $\D_{\F_2[F]}$-closure of $H$. In particular by \cref{proptor}, we have that
\begin{equation}\label {Euler=tor1}
\beta_0^{\F_2[F]}({}^FI_{\F_2[\widetilde H]}/{}^FI_{\F_2[H]})=0.
\end{equation}
 We want to show that 
$$\tor_{\Q[F]}(I_{\Q[F]}/{}^FI_{\Q[H]})={}^FI_{\Q[\widetilde H]}/{}^FI_{\Q[H]}.$$
This will follow from the following two claims.
\begin{claim} 
We have that $\beta_0^{\Q[F]}({}^FI_{\Q[\widetilde H]}/{}^FI_{\Q[H]})=0$. In particular, $$ {}^FI_{\Q[\widetilde H]}/{}^FI_{\Q[H]}\leq \tor_{\Q[F]}(I_{\Q[F]}/{}^FI_{\Q[H]}).$$

\end{claim}

\begin{proof}\renewcommand\qedsymbol{$\diamond$}
Consider the left $\Z[F]$-module   $N={}^FI_{\Z[\widetilde H]}/{}^FI_{\Z[H]}$.
Then we obtain that 
$$
\beta_0^{\Q[F]}(\Q\otimes_{\Z} N)\myle{\cref{Qp}}\beta_0^{\F_2[F]}(\F_2\otimes_{\Z} N)\myeq{\cref{Euler=tor1}} 0.$$
Now, \cref{proptor} implies the claim.
\end{proof}

\begin{claim} 
The left $\Q[F]$-module $I_{\Q[F]}/{}^FI_{\Q[\widetilde H]}$ is $\D_{\Q[F]}$-torsion-free. In particular, $$ \tor_{\Q[F]}(I_{\Q[F]}/{}^FI_{\Q[H]})\leq {}^FI_{\Q[\widetilde H]}/{}^FI_{\Q[H]}.$$
\end{claim}
\begin{proof}Put 
\(
M = I_{\mathbb{Z}[F]}/{}^{F}I_{\mathbb{Z}[\widetilde{H}]}.
\)
Assume that $\tor_{\mathbb{Z}[F]}(M) \neq \{0\}$. Observe that $M$ is free as a $\mathbb{Z}$-module and that $M / \tor_{\mathbb{Z}[F]}(M)$ is $\mathbb{Z}$-torsion-free. Thus, there exists $m \in \tor_{\mathbb{Z}[F]}(M)$ such that $m \notin 2M$.  
Put 
\[
\overline{M} = M / \mathbb{Z}[F]m.
\]

By \cref{Qp} we have
\begin{equation}\label{desQf2}
\beta_0^{\mathbb{Q}F}\!\left(\mathbb{Q} \otimes_{\mathbb{Z}} M\right) 
   \leq \beta_0^{\mathbb{F}_2 F}\!\left(\mathbb{F}_2 \otimes_{\mathbb{Z}} M\right),
   \quad\text{and}\quad
\beta_0^{\mathbb{Q}F}\!\left(\mathbb{Q} \otimes_{\mathbb{Z}} \overline{M}\right) 
   \leq \beta_0^{\mathbb{F}_2 F}\!\left(\mathbb{F}_2 \otimes_{\mathbb{Z}} \overline{M}\right).
\end{equation}

Since $m \in \tor_{\mathbb{Z}[F]}(M)$,
\begin{equation}
 \label{igQ}
\beta_0^{\mathbb{Q}F}\!\left(\mathbb{Q} \otimes_{\mathbb{Z}} M\right) 
   = \beta_0^{\mathbb{Q}F}\!\left(\mathbb{Q} \otimes_{\mathbb{Z}} \overline{M}\right),
\end{equation}
and since $m \notin 2M$ and $\mathbb{F}_2 \otimes_{\mathbb{Z}} \overline{M}$ is $\mathcal{D}_{\mathbb{F}_2[F]}$-torsion-free,
\begin{equation}
    \label{desF2}
\beta_0^{\mathbb{F}_2 F}\!\left(\mathbb{F}_2 \otimes_{\mathbb{Z}} M\right) 
   = \beta_0^{\mathbb{F}_2 F}\!\left(\mathbb{F}_2 \otimes_{\mathbb{Z}} \overline{M}\right) + 1.
\end{equation}

By Property (P3), 
\[
\beta_1^{\mathbb{F}_2[F]}\!\left(\mathbb{F}_2 \otimes_{\mathbb{Z}} M\right) = 0.
\] 
Thus, 
\begin{multline*}
\beta_0^{\mathbb{F}_2[F]}\!\left(\mathbb{F}_2 \otimes_{\mathbb{Z}} M\right) 
   = \chi^{\mathbb{F}_2[F]}\!\left(\mathbb{F}_2 \otimes_{\mathbb{Z}} M\right) =\operatorname{rk}(F) - \operatorname{rk}(\widetilde{H}) 
   = \\ \beta_0^{\mathbb{Q}F}\!\left(\mathbb{Q} \otimes_{\mathbb{Z}} M\right) 
     - \beta_1^{\mathbb{Q}F}\!\left(\mathbb{Q} \otimes_{\mathbb{Z}} M\right) 
   \leq \beta_0^{\mathbb{Q}F}\!\left(\mathbb{Q} \otimes_{\mathbb{Z}} M\right).
\end{multline*}

Therefore, taking into account \cref{desQf2}, we conclude that 
\begin{equation}
    \label{igQF2}
\beta_0^{\mathbb{F}_2[F]}\!\left(\mathbb{F}_2 \otimes_{\mathbb{Z}} M\right) 
   = \beta_0^{\mathbb{Q}F}\!\left(\mathbb{Q} \otimes_{\mathbb{Z}} M\right).
\end{equation}

Hence 

\begin{multline*}
\beta_0^{\mathbb{Q}F}\!\left(\mathbb{Q} \otimes_{\mathbb{Z}} \overline M\right) 
   \myeq{\cref{igQ}} \beta_0^{\mathbb{Q}F}\!\left(\mathbb{Q} \otimes_{\mathbb{Z}} M \right)
   \myeq{\cref{igQF2}}\\
   \beta_0^{\mathbb{F}_2[F]}\!\left(\mathbb{F}_2 \otimes_{\mathbb{Z}} M\right) 
   \myeq{\cref{desF2}} \beta_0^{\mathbb{F}_2[F]}\!\left(\mathbb{F}_2 \otimes_{\mathbb{Z}} \overline M\right) +1,
   \end{multline*}
 that contradicts \cref{desQf2}.
Therefore, $  \tor_{\Z[F]}(M)=0$.   Hence $ \Q\otimes_\Z M$ is  $\D_{\Q[F]}$-torsion-free.
 \renewcommand\qedsymbol{$\diamond$}\end{proof}

\subsection{Proof of   \cref{inertcompressed}} Let  $H$ be a finitely generated subgroup of $F$. In the introduction we have defined
$$
\bpi(H\leq F)=\min\{\rk(L)\colon H\le L\le F\} .$$
We obtain 
\cref{inertcompressed} as a consequence of the following result.
\begin{teo}\label{pibar}
Let $K$ be a field of characteristic 0 or 2, $F$  a free group and $H$ a finitely generated subgroup of $F$. Then 
\begin{equation}
\label{equalitypi}
\bpi(H\leq F) =\rk(H)-\beta_1^{K[F]}(I_{K[F]}/{}^FI_{K[H]}).
\end{equation}
Moreover, if $\widetilde H$ is the $\D_{K[F]}$-closure of $H$, then $\rk(\widetilde H)=\bpi(H\leq F) $.

\end{teo}
\begin{proof}
First observe that if $P$ is the prime subfield of  a field $K$, then, by  \cref{extfields},   
$$\beta_1^{K[F]}(I_{K[F]}/{}^F I_{K[H]})=\beta_1^{P[F]}(I_{P[F]}/{}^FI_{P[H]}).$$ Thus, it is enough to show the theorem in the case where $K$ is a  prime field.

Notice that   for a finitely generated subgroup $L$ of $F$ containing $H$, we have that
\begin{multline*}
\rk(L)=\rk(F)-\chi^{K[F]}(I_{K[F]}/{}^FI_{K[L]})=\\
\rk (F)-\beta_0^{K[F]}(I_{K[F]}/{}^FI_{K[L]})+
\beta_1^{K[F]}(I_{K[F]}/{}^FI_{K[L]})\geqslant \\ \rk (F)-\beta_0^{K[F]}(I_{K[F]}/{}^FI_{K[L]})\geqslant  \rk (F)-\beta_0^{K[F]}(I_{K[F]}/{}^FI_{K[H]}),\end{multline*}
and so, 
\begin{multline}\label{onedirection}
\bpi(H\leq F)\geqslant \rk (F)-\beta_0^{K[F]}(I_{K[F]}/{}^FI_{K[H]})= \\ 
\chi^{K[F]} (I_{K[F]}/{}^FI_{K[H]})+\rk(H)-\beta_0^{K[F]}(I_{K[F]}/{}^FI_{K[H]})=\\
\rk(H)-\beta_1^{K[F]}(I_{K[F]}/{}^FI_{K[H]}).
\end{multline}
This implies  one direction in \cref{equalitypi}.

Now we prove  the result  for $K=\Q$, and so, when $K$ is a field of characteristic 0. A similar argument gives the theorem for fields of characteristic 2.

Let $\widetilde H$ be the $L^2$-closure of $H$ in $F$. Then
\begin{multline}\label{inclosed}
\bpi(H\leq F)\leq \rk(\widetilde H)=\rk(F)-\chi^{\Q[F]}(I_{\Q[F]}/{}^FI_{\Q[\widetilde H]})=\\ \rk(F)-\beta_0(I_{\Q[F]}/{}^FI_{\Q[\widetilde H]})= \rk(F)-\beta_0(I_{\Q[F]}/{}^FI_{\Q[ H]})=\\ \rk(H)-\beta_1^{\Q[F]}(I_{\Q[F]}/{}^FI_{\Q[H]}).\qedhere
\end{multline}
 \end{proof}
 In the case of fields $K$ of characteristic $p\ne 0,2$ we do not know that free groups are $\D_{K[F]}$-subgroup rigid, and, therefore, the previous proof cannot applied in these cases.

 The following corollary is a strong version of \cref{inertcompressed}.
\begin{cor}\label{strcom}
Let $F$ be a free group and $H$ a finitely generated subgroup. Then the following are equivalent.
\begin{enumerate}
\item $H$ is compressed in $F$;
\item $H$ is inert in $F$;
\item $H$ is strongly inert in $F$;
\item $H$ is $L^2$-independent in $F$; 
\item for any field $K$ of characteristic $0$ or $2$, $H$ is $\D_{K[F]}$-independent in $F$.
\end{enumerate}
\end{cor}

\begin{proof}
The implications (5)$\Rightarrow$(4)$\Rightarrow$(3)$\Rightarrow$(2)$\Rightarrow$(1) have been discussed before. 

We will prove now (1)$\Rightarrow$(5). Let $K$ be a field of characteristic 0 or 2. If $H$ is compressed, then $\bpi(H\leq F)=\rk(H)$. By \cref{pibar}, $\beta_1^{K[F]}(I_{K[F]}/{}^FI_{K[H]})=0$. Hence $H$ is $\D_{K[F]}$-independent in $F$.
\end{proof}

\subsection{Proof of \cref{strictlycompressed}}
Let $H$ be a subgroup of $F$. We say that $H$ is \textbf{$\mathcal{D}_{K[F]}$-closed} if  
\[
\tor_{K[F]}\!\bigl(I_{K[F]}/{}^F I_{K[H]}\bigr) = \{0\}.
\]
By \cref{equivdef}, a subgroup of $F$ is $L^2$-closed if and only if it is $\mathcal{D}_{\mathbb{Q}[F]}$-closed.  
In this section, we will show the following stronger version of \cref{strictlycompressed}.

\begin{cor}
Let $F$  a free group and $H$ a finitely generated subgroup of $F$. Then the following are equivalent.
\begin{enumerate}
\item $H$ is strictly compressed in $F$;
\item $H$ is strictly inert in $F$;
\item  $H$ is $L^2$-closed in $F$;
\item for any field $K$ of characteristic 0 or 2, $H$ is $\D_{K[F]}$-closed.
\end{enumerate}
\end{cor}
\begin{proof}
The implications (2)$\Rightarrow$(1) and (4)$\Rightarrow$(3)  are  clear. 

Let us first show that (3) $\Rightarrow$ (2).  
Let $H$ be an $L^2$-closed (equivalently, $\mathcal{D}_{\mathbb{Q}[F]}$-closed) subgroup of $F$, and let $U$ be an arbitrary subgroup of $F$.
Observe that since the left $\Q[F]$-module $I_{\Q[F]}/{}^F  I_{\Q[H]}$ is $\D_{\Q[F]}$-torsion-free, it is also $\D_{\Q[U]}$-torsion-free as a left $\Q[U]$-module.
Thus, since the left $\Q[U]$-module $I_{\Q[U]}/{}^UI_{\Q[H\cap U]}$ is a submodule of $I_{\Q[F]}/ {}^FI_{\Q[H]}$, $I_{\Q[U]}/{}^UI_{\Q[H\cap U]}$  is also $\D_{\Q[U]}$-torsion-free.

Thus, in order to show (3)$\Rightarrow$(2), it is enough to prove (3)$\Rightarrow$(1). Let $L$ be a finitely generated subgroup of $F$ properly containing $H$. Then
\begin{multline*}
\rk(L)=\rk(F) -\chi^{\Q[F]}(I_{\Q[F]}/{}^FI_{\Q[L]})
=\\
\rk(F) -\chi^{\Q[F]}(I_{\Q[F]}/{}^FI_{\Q[H]})+
\chi^{\Q[F]}({}^FI_{\Q[L]}/{}^FI_{\Q[H]})=\\
\rk(H)+\chi^{\Q[F]}({}^FI_{\Q[L]}/{}^FI_{\Q[H]})
\myges{\cref{positiveEuler}} \rk(H).
\end{multline*}
Now let us show (1)$\Rightarrow$(4). Let $H$ be a strictly compressed subgroup of $F$. Let $\widetilde H$ be its $\D_{K[F]}$-closure (which does not depend on the field $K$).  If  $H\ne \widetilde H$, then we should have $\rk(\widetilde H)>\rk (H)$. But this contradicts \cref{pibar}. Hence $H= \widetilde H$.
 \end{proof}

\subsection{Proof of \cref{crit}}
By  \cref{strcom}, $L_1$ and $L_2$ are $L^2$-independent in $F$. In particular they are inert. Hence $\rk(L_1\cap L_2)\le \rk (L_1)$, and so, $L_1\cap L_2\in \Crit(H\le F)$.

On the other hand, if $L\in  \Crit(H\le F)$, then
\begin{multline*}
\beta_0^{\Q[F]}({}^FI_{\Q[L]}/{}^FI_{\Q[H]})=\chi^{\Q[F]}({}^FI_{\Q[L]}/{}^FI_{\Q[H]})
+\beta_1^{\Q[F]}({}^FI_{\Q[L]}/{}^FI_{\Q[H]})\myle{\cref{0tor1}}\\
 \rk(L)-\rk(H)+\beta_1^{\Q[F]}(I_{\Q[F]}/{}^FI_{\Q[H]})\myeq{\cref{pibar}} 0.
 \end{multline*}
Thus, if $L=\langle L_1,L_2\rangle$, then
$$\beta_0^{\Q[F]}({}^FI_{\Q[L]}/{}^FI_{\Q[H]})\le \beta_0^{\Q[F]}({}^FI_{\Q[L_1]}/{}^FI_{\Q[H]})+\beta_0^{\Q[F]}({}^FI_{\Q[L_2]}/{}^FI_{\Q[H]})=0.$$
and
\begin{equation}\label{LH}
\beta_1^{\Q[F]}({}^FI_{\Q[L]}/{}^FI_{\Q[H]})\myge{\cref{0tor1}}\beta_1^{\Q[F]}({}^FI_{\Q[L_1]}/{}^FI_{\Q[H]})=\beta_1^{\Q[F]}( I_{\Q[F]}/{}^FI_{\Q[H]}).\end{equation}
Therefore, 
\begin{multline*}
\rk(L)=\rk(H)+\chi^{\Q[F]}({}^FI_{\Q[L]}/{}^FI_{\Q[H]})=\\
\rk(H)-\beta_1^{\Q[F]}({}^FI_{\Q[L]}/{}^FI_{\Q[H]})\myeq{\cref{pibar}}\\
 \bpi (H\leq F)+\beta_1^{\Q[F]}( I_{\Q[F]}/{}^FI_{\Q[H]})-\beta_1^{\Q[F]}({}^FI_{\Q[L]}/{}^FI_{\Q[H]})\myle{\cref{LH}}  \bpi(H\le F).
\end{multline*}
Hence $L\in \Crit (H\leq F)$.
\qed

\section{Final comments}
The paper raises several questions, which we will gather in this section.
Our proof of $\D_{\F_2[F]}$-subgroup rigidity of a free group $F$ depends on a peculiarity of the left  $\F_2[F]$-module $I_{\F_2[F]}$ (\cref{maxsubm}). However, we strongly believe that there are other methods that work for all primes.
\begin{Conj}
Let $F$ be a free group and $p$ an odd prime. Then $ F$ is $\D_{\F_p[F]}$-subgroup rigid. 
\end{Conj}

The previous problem can be divided in two subproblems.
\begin{Conj}
Let $F$ be a free group and $p$ an odd prime. Let $H$ be finitely generated subgroup of $F$ and $\widetilde H$ its $L^2$-closure. Then 
\begin{enumerate}
\item $\beta_0^{\F_p[F]}(I_{\F_p[F]}/{}^FI_{\F_p[H]})=\beta_0^{\Q[F]}(I_{\Q[F]}/{}^FI_{\Q[H]})$ and
\item the left $\F_p[F]$-module $I_{\F_p[F]}/{}^FI_{\F_p[\widetilde H]}$ is $\D_{\F_p[F]}$-torsion-free.
\end{enumerate}
\end{Conj}

The invariant $\bpi(H\leq F)$ is related with the {\bf primitivity rank of $H$ in $F$} introduced by Puder and Parzanchevski \cite{PP15}:
$$\pi(H\leq F)=\min\{\rk(L): H\lneq L\le F,\ H \textrm{\ is not a free factor of\ }L\}.$$
It is clear that
$$    \bpi(H\leq F)=\min\{\pi(H\leq F),\rk(H)\}.
$$
By analogy, if $K$ is a field,  we can define
\begin{multline*}
    \pi_K(H\leq F)=\min\{\rk(N): {}^FI_{K[H]}\lneq N\le I_{K[F]},\\ {}^FI_{K[H]} \textrm{\ is not a free summand  of\ }N\}.\end{multline*}
The following conjecture is a variation of \cite[Conjecture 1.9]{EPS21}.
\begin{Conj}
Let $F$ be a free group and $H$ a finitely generated subgroup. Then for any field $K$, $\pi_K(H\leq F)=\pi(H\leq F)$.
\end{Conj}

As mentioned in the introduction, we anticipate that our methods can be extended to establish $L^2$-subgroup rigidity not only for free groups. To provide a specific focus, we put forth the following conjecture.
\begin{Conj}
Locally indicable groups are $L^2$-subgroup rigid.
\end{Conj}
On the other hand we expect that there are non $L^2$-subgroup rigid groups.
\begin{Problem}
Construct non $L^2$-subgroup rigid (torsion-free) groups.
\end{Problem}

\bibliographystyle{amsalpha}
\bibliography{biblio}

\end{document}